\documentclass[a4paper]{amsart}

\usepackage[english]{babel} 
\usepackage{amsmath,amsthm,amsfonts,amssymb,units}
\usepackage[sort]{cite}
\usepackage{xcolor,tikz,tikz-3dplot}

\newcommand{\comment}[1]{}

\theoremstyle{plain}
\newtheorem{theo}{Theorem}[section]

\newtheorem{cor}[theo]{Corollary}
\newtheorem{rem}[theo]{Remark}

\theoremstyle{definition}

\numberwithin{equation}{section}

\def\om{\Omega}
\def\p{\partial}
\def\reals{\mathbb{R}}
\def\realsnotz{\mathbb{R}\setminus\{0\}}
\def\Amat{{\rm A}}

\def\Cf{C_{\mathrm{\textsc{F}}}}
\def\MC{\frac{C_F^2}{\alpha}}

\def\chiom{\chi_{\om_0}}

\def\ul{\underline}
\def\ol{\overline}

\def\ft{\tilde{f}}
\def\ut{\tilde{u}}
\def\pt{\tilde{p}}
\def\qt{\tilde{q}}
\def\chat{\hat{c}}

\def\elems{\mathcal{T}}
\def\elem{T}

\newcommand{\set}[2]{\{#1\,\mid\,#2\}}

\DeclareMathOperator{\A}{A}
\def\As{\A^{*}}

\def\na{\nabla}
\DeclareMathOperator{\opdiv}{div}
\def\div{\opdiv}

\def\L{L}
\def\Lcrd{\L_{\mathrm{\textsc{1}}}}
\def\Lcrdd{\L_{\mathrm{\textsc{2}}}}

\def\M{\mathcal{M}}
\def\Mcrd{\M_{\mathrm{\textsc{1}}}}
\def\Mcrdd{\M_{\mathrm{\textsc{2}}}}
\def\Mcrdmaj{\M^+_{\mathrm{\textsc{3}}}}
\def\Mcrdmin{\M^-_{\mathrm{\textsc{3}}}}

\DeclareMathOperator{\cont}{\sf C}
\DeclareMathOperator{\lebesgue}{\sf L}
\DeclareMathOperator{\hilbert}{\sf H}
\DeclareMathOperator{\divergence}{\sf D}

\def\ciga{\cont^\infty_0}

\def\lo{\lebesgue^1}			
\def\lt{\lebesgue^2}			
\def\li{\lebesgue^\infty}	

\def\ho{\hilbert^1}			
\def\hoga{\hilbert^1_0}

\def\d{\divergence}			

\newcommand{\norm}[1]{|#1|}	
\newcommand{\normm}[1]{\|#1\|}
\newcommand{\normmm}[1]{|\!|\!|#1|\!|\!|}
\newcommand{\norms}[1]{|\![#1]\!|}

\newcommand{\normli}[1]{\norm{#1}_{\infty}}

\newcommand{\normlt}[1]{\norm{#1}}
\newcommand{\normltc}[1]{\norm{#1}_{c}}
\newcommand{\normltci}[1]{\norm{#1}_{c^{-1}}}
\newcommand{\normltA}[1]{\norm{#1}_{\Amat}}
\newcommand{\normltAi}[1]{\norm{#1}_{\Amat^{-1}}}
\newcommand{\normltcb}[1]{\norm{#1}_{c-\div b}}
\newcommand{\normltcbi}[1]{\norm{#1}_{(c-\div b)^{-1}}}
\newcommand{\normltb}[1]{\norm{#1}_{-\div b}}
\newcommand{\normltbi}[1]{\norm{#1}_{(-\div b)^{-1}}}

\newcommand{\normho}[1]{\norm{#1}_{\ho}}	
\newcommand{\normd}[1]{\norm{#1}_{\d}}	

\newcommand{\scp}[2]{\langle#1,#2\rangle}
\newcommand{\scplt}[2]{\scp{#1}{#2}}

\title[\sc Functional A Posteriori Error Equalities]
{\Large\sf Functional A Posteriori Error Control for Conforming Mixed Approximations of the Reaction-Convection-Diffusion Problem}
\author{Immanuel Anjam}
\address{Fakult\"at f\"ur Mathematik,
Universit\"at Duisburg-Essen, Campus Essen, Germany}
\email[Immanuel Anjam]{immanuel.anjam@uni-due.de}
\keywords{functional a posteriori error estimates, error equalities, 
mixed formulations, combined norm, reaction-convection-diffusion}
\subjclass{65N15}
\date{\today}
\thanks{The author would like to thank Dirk Pauly and Sergey Repin for discussions and suggestions related to the results exposed in this paper. This research was supported by the Emil Aaltonen Foundation.}

\begin{document}

\begin{abstract}
In this paper we show how to obtain the exact value of the global error of a conforming mixed approximation of the reaction-convection-diffusion problem. We operate in the framework of functional type a posteriori error control. The error is measured in a combined norm which takes into account both the primal and dual variables. Our main results state that the exact global error value of a conforming mixed approximation is given by a functional which includes only known quantities. The presented error equalities hold with certain restrictions on the reaction coefficient and the convection vector, namely, under the conditions when the solution operators of the corresponding problems are isometries. For the case where these restrictions are not satisfied we derive a two-sided error estimate.
\end{abstract}

\maketitle
\tableofcontents

\newpage


\section{Introduction}

The results presented in this paper are based on the framework of functional type a posteriori error control. These type estimates and equalities are valid for any conforming approximation and contain only global constants. For a detailed exposition of the theory see the books \cite{repinbookone} by Repin and \cite{NeittaanmakiRepin2004} by Repin and Neittaanm\"aki. For a more computational point of view see \cite{MaliRepinNeittaanmaki2014} by Mali, Repin, and Neittaanm\"aki.

In this paper we will consider only conforming approximations, and we will measure the error of our approximations in a combined norm, which includes the error of both the primal and the dual variable. This is especially useful for mixed methods where one calculates an approximation for both the primal and dual variables, see e.g. the book \cite{brezzifortinbookone} by Brezzi and Fortin. We call this approximation pair a mixed approximation. We note here that we consider more regular mixed approximations than in \cite{brezzifortinbookone}. However, for less regular approximations the results of this paper can still be used by the help of post-processing techniques (this is briefly commented upon in Section 5).

Functional a posteriori error estimates for mixed approximations in combined norms were first exposed in the paper \cite{repinsautersmolianskiaposttwosideell}, where the authors studied real valued problems of the type
\begin{equation*}
	\As\A u=f,
\end{equation*}
where $\A$ is a differential operator, and $\As$ its adjoint. Here the primal variable is $u$, and the dual variable is $p:=\A u$. The simplest partial differential equation contained in this class of problems is the Poisson equation with $\A=\na$ and $\As=-\div$. In \cite{repinsautersmolianskiaposttwosideell} the authors present two-sided estimates bounding the combined error by the same quantity from below and from above aside from multiplicative constants.

Complex (and real) valued problems with lower order terms, i.e., of the following two types
\begin{equation*}
	\As\A u + u=f \qquad \textrm{and}
	\qquad
	\As\A u + i\nu u=f
\end{equation*}
were studied in \cite{anjampaulyeq}. Here $i$ is the imaginary unit and $\nu\in\realsnotz$. For example, the  reaction-diffusion equation belongs to the first class of problems, and the eddy-current problem to the second class of problems. In \cite{anjampaulyeq} the authors derive an error equality for the first class of problems, meaning that the exact global error value of a mixed approximation can be calculated by having the problem data and the approximation at hand. In the real valued case the error equality is a special case of the very general result \cite[(7.2.14)]{NeittaanmakiRepin2004}, and was also found in \cite[Remark 6.12]{caiestimators} in the case of the real valued reaction-diffusion problem. For the second type of problems the authors derive in \cite{anjampaulyeq} a two-sided estimate bounding the combined error by the same quantity from below and from above aside from multiplicative constants.

Most work on functional type a posteriori error control for the reaction-convection-diffusion problem has concentrated on error estimation for approximations of the primal variable. The error estimates resulting from this research can be found in \cite{eigelsamrowskiapost,repindivapost,nicaiserepinapost,kuzminhannukainenapost}, and some of these results are also exposed in the book \cite{repinbookone}. In \cite{nicaiserepinapost} the authors also derive two-sided estimates for mixed approximations in combined norms. The error measure is different to the one used in this paper, and is commented upon in Remark \ref{rem:CRD1}.

This paper concentrates in deriving error equalities and two-sided error estimates for mixed approximations of the reaction-convection-diffusion problem. We work in the complex valued setting, but the results hold also in the real case. The paper is organized as follows. In Section \ref{sec:MODEL} we state our model problem in two equivalent forms. We show that the solution operators of these problems are isometries, and comment on the fact that error control immediately follows if the primal approximation contains additional regularity. Section \ref{sec:APOST} is dedicated to a posteriori error equalities for the reaction-convection-diffusion equations. The main results of this section are stated in Theorems \ref{thm:CRD1} and \ref{thm:CRD2}. In Section \ref{sec:APOSTEST} we derive a two-sided error estimate which is applicable also in the case where both the reaction coefficient and convection vector divergence are zero simultaneously in some part of the domain. The main result of this section is stated in Theorem \ref{thm:CRD3}. In Section 5 we briefly comment on some aspects of the derived results.


\section{Notation and the Model Problems} \label{sec:MODEL}

In this paper all spaces are defined over the complex field, and if not otherwise indicated, functions are complex valued. All results hold unchanged in the real valued case as well.

Let $\om\subset\reals^d$, $d\geq1$, be an arbitrary (bounded or unbounded) domain. We denote by $\scplt{\,\cdot\,}{\,\cdot\,}_{\omega}$ and $\normlt{\,\cdot\,}_{\omega}$ the inner product and norm for scalar or vector valued functions in $\lt(\omega)$ for $\omega\subset\Omega$. We introduce the following subindex notation: we define $\scp{\,\cdot\,}{\,\cdot\,}_{\omega,\gamma}:=\scplt{\gamma\,\cdot\,}{\,\cdot\,}_{\omega}$, which induces $\norm{\,\cdot\,}_{\omega,\gamma}$, where $\gamma$ belongs to the space of essentially bounded functions $\li(\omega)$. If $\gamma$ is self-adjoint and uniformly positive definite, they become an inner product and a norm in $\lt(\omega)$, respectively. We denote by $\norm{\cdot}_{\omega,\infty}$ the norm in $\li(\omega)$. If $\omega=\om$, we will not indicate the dependence on domain in our notations, i.e., $\normlt{\,\cdot\,} = \normlt{\,\cdot\,}_{\om}$ for functions in $\lt = \lt(\om)$.

We define the usual Sobolev spaces
\begin{equation*}
	\ho:=\set{\varphi\in\lt}{\na\varphi\in\lt},
	\qquad
	\d:=\set{\psi\in\lt}{\div\psi\in\lt} .
\end{equation*}
These are Hilbert spaces equipped with the respective graph norms $\normho{\,\cdot\,}$, $\normd{\,\cdot\,}$. The space of functions belonging to $\ho$ and vanishing on the boundary is defined as the closure of smooth and compactly supported test functions
\begin{equation*}
	\hoga:=\ol{\ciga}^{\ho} .
\end{equation*}
We have
\begin{equation} \label{eq:partint}
	\forall\,\varphi\in\hoga\quad\forall\,\psi\in\d\qquad
	\scplt{\na\varphi}{\psi}=-\scplt{\varphi}{\div\psi} .
\end{equation}

The reaction-convection-diffusion problem in the mixed form consists of finding a scalar potential $u\in\ho$ and a diffusive flux $p\in\d$, such that
\begin{equation} \label{eq:pde1}
\begin{array}{r@{$\;$}c@{$\;$}l l}
	p & = & \Amat \na u & \quad \textrm{in } \om , \\
	-\div p + b \cdot \na u + c\,u & = & f & \quad \textrm{in } \om , \\
	u & = & 0 & \quad \textrm{on } \p\om ,
\end{array}
\end{equation}
where the source $f$ belongs to $\lt$. The diffusion matrix $\Amat\in\li$ is self-adjoint and uniformly positive definite. In particular
\begin{equation} \label{eq:Amat}
	\forall\,\varphi\in\lt\qquad
	\alpha \norm{\varphi}^2 \leq \scplt{\Amat\varphi}{\varphi} \leq \beta \norm{\varphi}^2 ,
\end{equation}
where $0 < \alpha \leq \beta$. The convection vector $b$ is a real valued function from $\li$ such that $\div b \in \li$. The real scalar valued reaction coefficient $c$ belongs to $\li$. The variables $u$ and $p$ are often called primal and dual variables, respectively.

By multiplying \eqref{eq:pde1} with a test function, and using \eqref{eq:partint} we obtain the weak formulation of this problem: find $u\in\hoga$ such that
\begin{equation} \label{eq:gen1}
	\forall\,w\in\hoga\qquad
	\scplt{\Amat \na u}{\na w} + \scplt{b \cdot \na u + c\,u}{w} =\scplt{f}{w}.
\end{equation}
For any $v,w\in\hoga$ we have the identity (see Appendix \ref{app:identity1})
\begin{equation} \label{eq:identity1}
	\scplt{b\cdot \na v}{w} = - \scplt{b\,v}{\na w} - \scplt{(\div b)v}{w} .
\end{equation}
By setting $v = w$ we see that
\begin{equation} \label{eq:convmanip}
	\normltb{w}^2 = \scplt{(-\div b) w}{w} = 2 \Re \scplt{b \cdot \na w}{w} .
\end{equation}
By setting $w=u$ and applying \eqref{eq:convmanip} to the real part of the form generated by the left hand side of \eqref{eq:gen1}, we see that
\begin{align}
	\Re\big( \scplt{\Amat \na u}{\na u} + \scplt{b \cdot \na u + c\,u}{u} \big)
	& = \scplt{\Amat \na u}{\na u} + \scplt{\left(c-\frac{1}{2}\div b\right)u}{u} \nonumber \\
	& = \scplt{\na u}{\na u}_{\Amat} + \scplt{u}{u}_{c-\frac{1}{2}\div b} \label{eq:coer} \\
	& = \normlt{\na u}^2_{\Amat} + \normlt{u}^2_{c-\frac{1}{2}\div b} . \nonumber
\end{align}
The form is then coercive provided that
\begin{equation*}
	c - \frac{1}{2} \div b \ge \lambda > 0
\end{equation*}
holds, and then \eqref{eq:gen1} has a unique solution $u\in\hoga$ by the Lax-Milgram Theorem. If the domain is bounded (by bounded we mean bounded at least in one direction, i.e., it lies between two parallel hyperplanes) this condition can be weakened by using the Friedrichs inequality
\begin{equation} \label{eq:Cf}
	\forall w \in \hoga \qquad
	\normlt{w} \leq \Cf \normlt{\na w},
\end{equation}
where $\Cf$ is the Friedrichs constant. We note that generally the exact value of the Friedrichs constant is unknown, but it is easy to estimate it from above. Using \eqref{eq:Amat} and \eqref{eq:Cf} we can write \eqref{eq:coer} as
\begin{align}
	\Re \big( \scplt{\Amat \na u}{\na u} + \scplt{b \cdot \na u + c\,u}{u} \big)
	& = \frac{1}{2} \normlt{\na u}^2_{\Amat} + \frac{1}{2} \normlt{\na u}^2_{\Amat} + \normlt{u}^2_{c-\frac{1}{2}\div b} \nonumber \\
	& \ge \frac{1}{2} \normlt{\na u}^2_{\Amat} + \frac{\alpha}{2\Cf^2} \normlt{u}^2 + \normlt{u}^2_{c-\frac{1}{2}\div b} \label{eq:coer2} \\
	& = \frac{1}{2} \normlt{\na u}^2_{\Amat} + \normlt{u}^2_{c-\frac{1}{2}\div b+\frac{\alpha}{2\Cf^2}} . \nonumber
\end{align}
The form is then coercive provided that
\begin{equation*}
	c-\frac{1}{2}\div b+\frac{\alpha}{2\Cf^2} \ge \lambda > 0
\end{equation*}
holds, and then \eqref{eq:gen1} has a unique solution $u\in\hoga$ by the Lax-Milgram Theorem.

By applying \eqref{eq:identity1} to \eqref{eq:gen1} we obtain an alternative formulation for the reaction-convection-diffusion problem:
\begin{equation*}
	\forall\,w\in\hoga\qquad
	\scplt{\Amat \na u - b\,u}{\na w} + \scplt{(c-\div b)u}{w} =\scplt{f}{w} ,
\end{equation*}
from which we can deduce the Euler equations in the mixed form:
\begin{equation} \label{eq:pde2}
\begin{array}{r@{$\;$}c@{$\;$}l l}
	q & = & \Amat \na u - b\,u & \quad \textrm{in }\om , \\
	-\div q + (c-\div b)u & = & f & \quad \textrm{in }\om , \\
	u & = & 0 & \quad \textrm{on }\p\om .
\end{array}
\end{equation}
The dual variable $q$ is the so-called total flux, the combination of the diffusive flux $p = \Amat \na u$ and the convective flux $bu$.

\begin{rem} \label{rem:isometry}
With additional assumptions on the material coefficients and 
equipped with suitable norms the solution operators of \eqref{eq:pde1} and \eqref{eq:pde2} are isometries.
\begin{itemize}
\item[\bf(i)] Assume $c \ge c_0 > 0$ and $c-\div b \geq 0$. Then the solution operator
\begin{equation*}
	\Lcrd :\lt\to\hoga\times\d;f\mapsto(u,p)
\end{equation*}
of \eqref{eq:pde1} is an isometry, i.e., $\norm{\Lcrd}=1$: from \eqref{eq:pde1} we see that
\begin{align*}
	\normltci{f}^2
	& = \normltci{c\,u + b\cdot\na u - \div p}^2 \\
	& = \normltci{c\,u}^2 + \normltci{b\cdot\na u - \div p}^2 + 2\Re\scp{c\,u}{b\cdot\na u - \div p}_{c^{-1}} \\
	& = \normltc{u}^2 + \normltci{b\cdot\na u - \div p}^2 + 2\Re\scplt{u}{b\cdot\na u - \div p} .
\end{align*}
Together with $\scplt{u}{-\div p} = \scplt{\na u}{p} = \scplt{\na u}{\Amat \na u} = \normltA{\na u}^2 = \normltAi{p}^2 $ and \eqref{eq:convmanip} this becomes
\begin{align*}
	\normltci{f}^2
	& = \normltc{u}^2 + \normltci{b\cdot\na u - \div p}^2 + \normltA{\na u}^2 + \normltAi{p}^2 + \normltb{u}^2 \\
	& = \normltcb{u}^2 + \normltA{\na u}^2 + \normltAi{p}^2 + \normltci{b\cdot\na u - \div p}^2 \\
	& =: \normm{(u,p)}^2 .
\end{align*}
\item[\bf(ii)] Assume $c \ge 0$ and $c-\div b \ge \lambda > 0$. Then the solution operator
\begin{equation*}
	\Lcrdd :\lt\to\hoga\times\d;f\mapsto(u,q)
\end{equation*}
of \eqref{eq:pde2} is an isometry, i.e., $\norm{\Lcrdd}=1$: from \eqref{eq:pde2} we see that
\begin{align*}
	\normltcbi{f}^2
	& = \normltcbi{(c-\div b)u - \div q}^2 \\
	& = \normltcbi{(c-\div b)u}^2 + \normltcbi{\div q}^2 + 2\Re\scp{(c-\div b)u}{-\div q}_{(c-\div b)^{-1}} \\
	& = \normltcb{u}^2 + \normltcbi{\div q}^2 + 2\Re\scplt{u}{-\div q} .
\end{align*}
Together with $\scplt{u}{-\div q} = \scplt{\na u}{q} = \scplt{\na u}{\Amat\na u - b\,u} = \normltA{\na u}^2 - \scplt{\na u}{b\,u}$ and \eqref{eq:convmanip} this becomes
\begin{align*}
	\normltcbi{f}^2
	& = \normltcb{u}^2 + \normltcbi{\div q}^2 + 2\normltA{\na u}^2 - \normltb{u}^2 \\
	& = \normltc{u}^2 + \normltA{\na u}^2 + \normltAi{q+b\,u}^2 + \normltcbi{\div q}^2 \\
	& =: \normmm{(u,q)}^2 ,
\end{align*}
where we have also used $\normltA{\na u} = \normltA{\Amat^{-1}(q +b\,u)} = \normltAi{q +b\,u}$.
\end{itemize}
\end{rem}

It is easy to see that the norms $\normm{\cdot}$ and $\normmm{\cdot}$ are equivalent to weighed $\ho \times \d$ -norms (see Appendix \ref{app:normequiv}), giving some justification on their use as error measures. In fact, Remark \ref{rem:isometry} immediately furnishes error equalities for sufficiently regular approximations: Let the primal approximation $\ut\in\hoga$ be such that $\pt=\Amat\na\ut \in \d$. Then, the pair $(\ut,\pt)$ can be seen as the exact solution of the problem
\begin{equation*}
\begin{array}{r@{$\;$}c@{$\;$}l l}
	\pt & = & \Amat \na \ut & \quad \textrm{in } \om , \\
	-\div\pt + b \cdot \na \ut + c\,\ut & =: & \ft & \quad \textrm{in } \om , \\
	\ut & = & 0 & \quad \textrm{on } \p\om ,
\end{array}
\end{equation*}
i.e., we have $\Lcrd(\ft) = (\ut,\pt)$.
Then by Remark \ref{rem:isometry}(i) we directly have
\begin{equation*}
	\normm{(u,p)-(\ut,\pt)}^2 = \normm{\Lcrd(f-\ft)}^2 = \normltci{f-\ft}^2
	= \normltci{f - c\,\ut - b \cdot \na \ut + \div\pt}^2
\end{equation*}
and this can obviously be normalized to obtain
\begin{equation*}
	\frac{\normm{(u,p)-(\ut,\pt)}^2}{\normm{(u,p)}^2}
	= \frac{\normltci{f - c\,\ut - b \cdot \na \ut + \div\pt}^2}{\normltci{f}^2} .
\end{equation*}
A similar error equality for the formulation \eqref{eq:pde2} follows directly from Remark \ref{rem:isometry}(ii). The assumption of the high regularity of $\ut$ is not very practical, and we will remove it in subsequent sections. However, we emphasize that for linear problems error control follows directly from estimates of the norm of the solution operator, provided that the primal approximation is sufficiently regular.


\section{Error Equalities} \label{sec:APOST}

In this section we derive error equalities for the two formulations of the reaction-convection-diffusion problem discussed in the previous section.

In the following, we will understand the pairs $(\ut,\pt),(\ut,\qt)\in\hoga\times\d$ without further requirements as (mixed) approximations of the exact solutions $(u,p),(u,q)\in\hoga\times\d$, respectively. In other words, the results are applicable to any \emph{conforming} approximation, i.e., approximations which belong to the appropriate Hilbert spaces and satisfy the boundary conditions exactly. To emphasize further, we do not attract any specific properties of a numerical method.

In the previous section we listed the requirements for the material coefficients $\Amat,b,$ and $c$ which guarantee existence and uniqueness of a solution. For clarity we will not repeat these requirements in this section, but simply list those requirements which are needed for the error equalities to hold. These requirements will be same as the ones in Remark \ref{rem:isometry}, i.e., the conditions under which the solution operators are isometries.


\subsection{Error Equality for the Problem \eqref{eq:pde1}}

We present the first main result of the paper.

\begin{theo} 
\label{thm:CRD1}
Let $(u,p)\in\hoga\times\d$ be the exact solution of \eqref{eq:pde1}, and assume $c \ge c_0 > 0$ and $c-\div b \geq 0$. Let $(\ut,\pt)\in\hoga\times\d$ be arbitrary. Then
\begin{equation} \label{eq:CRD1_main1}
	\normm{(u,p)-(\ut,\pt)}^2 = \Mcrd(\ut,\pt)
\end{equation}
and the normalized counterpart
\begin{equation} \label{eq:CRD1_main2}
	\frac{\normm{(u,p)-(\ut,\pt)}^2}{\normm{(u,p)}^2}
	= \frac{\Mcrd(\ut,\pt)}{\normltci{f}^2}
\end{equation}
hold, where
\begin{align*}
	\normm{(u,p)-(\ut,\pt)}^2 = & \, \normltcb{u-\ut}^2 + \normltA{\na(u-\ut)}^2 +
	\normltAi{p-\pt}^2 + \normltci{b\cdot\na(u-\ut)-\div(p-\pt)}^2 , \\
	\Mcrd(\ut,\pt) := &  \, \normltci{f - c\,\ut - b\cdot\na\ut + \div\pt}^2
	+ \normltAi{\pt - \Amat\na\ut}^2 .
\end{align*}
\end{theo}

\begin{proof}
We begin by using \eqref{eq:pde1} and $p=\Amat \na u$:
\begin{align*}
\Mcrd(\ut,\pt)
&=\normltci{f - c\,\ut - b\cdot\na\ut + \div\pt}^2 + \normltAi{\pt-\Amat\na\ut}^2\\
&=\normltci{c(u-\ut) + b\cdot\na(u-\ut) + \div(\pt-p)}^2+\normltAi{\pt-p + \Amat\na(u-\ut)}^2\\
&=\normltci{c(u-\ut)}^2+\normltci{b\cdot\na(u-\ut) + \div(\pt-p)}^2
+2\Re\scp{c(u-\ut)}{b\cdot\na(u-\ut) + \div(\pt-p)}_{c^{-1}}\\
&\qquad+\normltAi{\pt-p}^2+\normltAi{\Amat\na(u-\ut)}^2+2\Re\scp{\pt-p}{\Amat\na(u-\ut)}_{\Amat^{-1}}\\
&=\normltc{u-\ut}^2 + \normltA{\na(u-\ut)}^2 + \normltAi{\pt-p}^2 + \normltci{b\cdot\na(u-\ut) + \div(\pt-p)}^2\\
&\qquad+2\Re\scplt{\pt-p}{\na(u-\ut)}
+2\Re\scplt{u-\ut}{\div(\pt-p)}
+2\Re\scplt{u-\ut}{b\cdot\na(u-\ut)} .
\end{align*}
The first two terms in the last line cancel each other due to \eqref{eq:partint}, since $u-\ut \in \hoga$ and $\pt-p \in \d$. By using \eqref{eq:convmanip} we see that the third term in the last line can be written as
\begin{equation*}
	2\Re\scplt{u-\ut}{b\cdot\na(u-\ut)} = \normltb{u-\ut}^2 .
\end{equation*}
We arrive at \eqref{eq:CRD1_main1} by combining this term to the other weighed $\lt$ norm of the difference $u-\ut$. The equality \eqref{eq:CRD1_main2} follows by the isometry property in Remark \ref{rem:isometry}(i).
\end{proof}

\begin{cor}
\label{cor:diffb}
If $b = 0$, the problem \eqref{eq:pde1} becomes the reaction-diffusion equation. For the error equality to hold we need to assume $c \ge c_0 > 0$. The equality of Theorem \ref{thm:CRD1} then becomes
\begin{equation*}
	\normltc{u-\ut}^2 + \normltA{\na(u-\ut)}^2 +
	\normltAi{p-\pt}^2 + \normltci{\div(p-\pt)}^2 \\
	= \normltci{f - c\,\ut + \div\pt}^2 + \normltAi{\pt - \Amat\na\ut}^2 ,
\end{equation*}
which is the result from \cite{anjampaulyeq,NeittaanmakiRepin2004,caiestimators}. Naturally, this error equality has a normalized counterpart.
\end{cor}

\begin{rem} \label{rem:CRD1}
We note the following:
\begin{itemize}
\item[\bf(i)] If $c-\div b=0$ boundedness of the domain is needed for proper error control (see Appendix \ref{app:normequiv}).
\item[\bf(ii)] One of the terms in the error measure can be written as
\begin{equation*}
	\normltci{b\cdot\na(u-\ut)-\div(p-\pt)}^2 = \normltci{f - c\,u - b\cdot\na\ut+\div\pt}^2 .
\end{equation*}
\item[\bf(iii)] By applying the first norm equivalence in Appendix \ref{app:normequiv} to the equality of Theorem \ref{thm:CRD1} one obtains a two-sided estimate where the error is measured in a weighed $\ho\times\d$ -norm. This two-sided estimate is similar to the one derived in\cite{nicaiserepinapost}. However, these bounds can be quite coarse, which is apparent from the constants appearing in the norm equivalence.
\end{itemize}
\end{rem}

We emphasize here that the error equality is possible because of the presence of the reactive term $c \, u$. Indeed, $c$ cannot be zero. Unfortunately this means that the result of Theorem \ref{thm:CRD1} cannot be used for the convection-diffusion problem. However, the result of the next section allows also for this case.


\subsection{Error Equality for the Problem \eqref{eq:pde2}}

We present the second main result of the paper.

\begin{theo} 
\label{thm:CRD2}
Let $(u,q)\in\hoga\times\d$ be the exact solution of \eqref{eq:pde2}, and assume $c \ge 0$ and $c-\div b \ge \lambda > 0$. Let $(\ut,\qt)\in\hoga\times\d$ be arbitrary. Then
\begin{equation} \label{eq:CRD2_main1}
	\normmm{(u,q)-(\ut,\qt)}^2 = \Mcrdd(\ut,\qt)
\end{equation}
and the normalized counterpart
\begin{equation} \label{eq:CRD2_main2}
	\frac{\normmm{(u,q)-(\ut,\qt)}^2}{\normmm{(u,q)}^2}
	= \frac{\Mcrdd(\ut,\qt)}{\normltcbi{f}^2}
\end{equation}
hold, where
\begin{align*}
	\normmm{(u,q)-(\ut,\qt)}^2 = & \, \normltc{u-\ut}^2 + \normltA{\na(u-\ut)}^2 +
	\normltAi{q-\qt + b(u-\ut)}^2 + \normltcbi{\div(q-\qt)}^2 , \\
	\Mcrdd(\ut,\qt) := & \, \normltcbi{f - (c-\div b)\ut + \div\qt}^2
	+\normltAi{\qt - \Amat\na\ut + b\,\ut}^2.
\end{align*}
\end{theo}

\begin{proof}
We begin by using \eqref{eq:pde2} and $q=\Amat \na u - b\,u$:
\begin{align*}
\Mcrdd(\ut,\qt)
&=\normltcbi{f - (c-\div b)\ut + \div\qt}^2+\normltAi{\qt - \Amat\na\ut + b\,\ut}^2 \\
&=\normltcbi{(c-\div b)(u-\ut) + \div(\qt-q)}^2 + \normltAi{\qt-q + \Amat\na(u-\ut) + b(\ut-u)}^2 \\
&=\normltcbi{(c-\div b)(u-\ut)}^2+\normltcbi{\div(\qt-q)}^2\\
&\qquad+2\Re\scp{(c-\div b)(u-\ut)}{\div(\qt-q)}_{(c-\div b)^{-1}} \\
&\qquad+\normltAi{\qt-q+b(\ut-u)}^2+\normltAi{\Amat\na(u-\ut)}^2
+2\Re\scp{\qt-q+b(\ut-u)}{\Amat\na(u-\ut)}_{\Amat^{-1}} \\
&=\normltcb{u-\ut}^2 + \normltA{\na(u-\ut)}^2 + \normltAi{\qt-q+b(\ut-u)}^2 + \normltcbi{ \div(\qt-q)}^2 \\
&\qquad+2\Re\scplt{\qt-q}{\na(u-\ut)}
+2\Re\scplt{u-\ut}{\div(\qt-q)}
-2\Re\scplt{u-\ut}{b\cdot\na(u-\ut)} .
\end{align*}
The first two terms in the last line cancel each other due to \eqref{eq:partint}, since $u-\ut \in \hoga$ and $\qt-q \in \d$. By using \eqref{eq:convmanip} we see that the third term in the last line can be written as
\begin{equation*}
	- 2\Re\scplt{u-\ut}{b\cdot\na(u-\ut)} = \norm{u-\ut}^2_{\div b} .
\end{equation*}
We arrive at \eqref{eq:CRD2_main1} by combining this term to the other weighed $\lt$ norm of the difference $u-\ut$. The equality \eqref{eq:CRD2_main2} follows by the isometry property in Remark \ref{rem:isometry}(ii).
\end{proof}

Note that we can set $c=0$ in the result of Theorem \ref{thm:CRD2}, because there is `additional reaction' provided by the term $-(\div b)u$. This corresponds to the case of the convection-diffusion problem. However, if $c=0$, then the result is not valid if $\div b=0$. This means that $b$ cannot be zero in this case, or even a constant vector. The case $c=\div b=0$ will be considered in the next section.

\begin{cor}
\label{cor:diffc}
If $c = 0$, the problem \eqref{eq:pde2} becomes the convection-diffusion equation. For the error equality to hold we need to assume $-\div b \ge b_0 > 0$. The equality of Theorem \ref{thm:CRD2} then becomes
\begin{multline*}
	\normltA{\na(u-\ut)}^2 + \normltAi{q-\qt + b(u-\ut)}^2 + \normltbi{\div(q-\qt)}^2 \\
	= \normltbi{f + (\div b)\ut + \div\qt}^2
		+ \normltAi{\qt - \Amat\na\ut + b\,\ut}^2 .
\end{multline*}
Note that for proper error control boundedness of the domain is now required (see Appendix \ref{app:normequiv}). Naturally, this error equality has a normalized counterpart.
\end{cor}

\begin{rem} We note the following:
\begin{itemize}
\item[\bf(i)] One of the terms in the error measure can be written as
\begin{equation*}
	\normltAi{q-\qt + b(u-\ut)}^2 = \normltAi{\Amat\na u-\qt-b\,\ut}^2 .
\end{equation*}
\item[\bf(ii)] By applying the second norm equivalence in Appendix \ref{app:normequiv} to the result of Theorem \ref{thm:CRD2} one obtains a two-sided estimate in a weighed $\ho\times\d$ -norm. However, these bounds can be quite coarse, which is apparent from the constants appearing in the norm equivalence.
\end{itemize}
\end{rem}


\section{Two-Sided Error Estimates for the Case $c=\div b=0$} \label{sec:APOSTEST}

If $c=0$ and $\div b=0$ simultaneously in some part of the domain, neither of the error equalities of the previous section can be used. For this case we derive a two-sided error bound by using the error equality of Theorem \ref{thm:CRD1}.

In this section, we consider the problem formulation \eqref{eq:pde1} for a bounded domain $\om$ such that the Friedrichs inequality \eqref{eq:Cf} holds. Then we see from \eqref{eq:coer2} that even in the case $c=\div b=0$ the form generated by the left hand side of \eqref{eq:gen1} is coercive so we have a unique solution $u\in\hoga$ by the Lax-Milgram Theorem.

Like before, the real scalar valued reaction coefficient $c\ge0$ belongs to $\li$. We denote by $\om_0$ the part of the domain where $c=0$. In the other part $\om_c:=\om\setminus\om_0$ we assume that $c$ is uniformly positive definite, i.e., $c\ge c_0>0$. We define $\chat := c + \chiom$, where $\chiom$ (defined on the whole domain $\om$) is the characteristic function of $\om_0$. See Figure \ref{fig:division} for an example. Note that $\chat$ belongs to $\li$ and is uniformly positive definite.

\begin{figure}
\centering
{\footnotesize
\begin{tikzpicture}[scale=0.9]
	\fill [gray,opacity=.15] (1,1.7) .. controls (2,0) and (4,1) .. (5,1.0) .. controls (5.5,1) and (6,0.5) .. (6.5,-0.5) .. controls (-0.5,-1) and (1,2) .. (1,1.7);
	\draw [line width=1pt] (1,1.7) .. controls (2,0) and (4,1) .. (5,1.0) .. controls (5.5,1) and (6,0.5) .. (6.5,-0.5) .. controls (-0.5,-1) and (1,2) .. (1,1.7);
	\draw [dashed,line width=1pt]  (1,1.7) .. controls (2,0) and (3,0.2) .. (6.5,-0.5);
	\fill [white] (5,0.2) circle (0.55);				
	\draw [line width=1pt] (5,0.2) circle (0.55);
	\draw (3,1) node [] {$\om$};
	\draw (2.32,0.099) node [white] {$\om_0$}; 
	\draw (2.3,0.1) node [] {$\om_0$};
	\draw (4.02,0.399) node [white] {$\om_c$}; 
	\draw (4,0.4) node [] {$\om_c$};
	\draw [] (7,0.5) .. controls (6.4,0.6) .. (5.75,0.3);
	\draw (7.9,0.3) node [] {$	\begin{array}{r@{$\;$}l}
									c & \ge c_0 > 0 \\
									\chiom & = 0
								\end{array}$ };
	\draw [] (0.4,0.3) .. controls (0.9,0.2) .. (1.5,0.45);
	\draw (-0.5,0.3) node [] {$ 	\begin{array}{r@{$\;$}l}
									c & = 0 \\
									\chiom & = 1
								\end{array}$ };
	\draw [gray] (9.5,1.7) -- (9.5,-0.7);
\end{tikzpicture}
\hskip-2em
\tdplotsetmaincoords{80}{100}
\begin{tikzpicture}[scale=0.6,tdplot_main_coords]
	\fill [gray,opacity=.5] (0.5,1,1.7) .. controls (1,2,0) and (1,3,0.2) .. (1,6.5,-0.5) -- (0,6.5,-0.5) .. controls (0,3,0.2) and (0,2,0) .. (0.5,1,1.7);
	\draw [dashed] (0.5,1,1.7) .. controls (1,2,0) and (1,3,0.2) .. (1,6.5,-0.5);
	\draw [dashed]  (0.5,1,1.7) .. controls (0,2,0) and (0,3,0.2) .. (0,6.5,-0.5);
	\fill [white] (0,5,0.75) .. controls (0,4.3,0.75) and (0,4.3,-0.35) .. (0,5,-0.35) .. controls (0,5.7,-0.35) and (0,5.7,0.75) .. (0,5,0.75);	
	\fill [white] (1,5,0.75) .. controls (1,4.3,0.75) and (1,4.3,-0.35) .. (1,5,-0.35) .. controls (1,5.7,-0.35) and (1,5.7,0.75) .. (1,5,0.75);
	
	\draw [line width=1pt] (0.5,1,1.7) .. controls (0,2,0) and (0,4,1) .. (0,5,1.0) .. controls (0,5.5,1) and (0,6,0.5) .. (0,6.5,-0.5);
	\draw [gray] (0,6.5,-0.5) .. controls (0,-0.5,-1) and (0,1,2) .. (0.5,1,1.7);
	\draw [gray] (0,5,0.75) .. controls (0,4.3,0.75) and (0,4.3,-0.35) .. (0,5,-0.35) .. controls (0,5.7,-0.35) and (0,5.7,0.75) .. (0,5,0.75);
	\draw [black,line width=1pt] (0,4.56,0.5) .. controls (0,4.40,0.4) and (0,4.39,-0.57) .. (0,5.25,-0.3);
	\draw [line width=1pt] (0.5,1,1.7) .. controls (1,2,0) and (1,4,1) .. (1,5,1.0) .. controls (1,5.5,1) and (1,6,0.5) .. (1,6.5,-0.5) .. controls (1,-0.5,-1) and (1,1,2) .. (0.5,1,1.7);
	\draw [line width=1pt] (1,5,0.75) .. controls (1,4.3,0.75) and (1,4.3,-0.35) .. (1,5,-0.35) .. controls (1,5.7,-0.35) and (1,5.7,0.75) .. (1,5,0.75);
	\draw [line width=1pt] (0,6.5,-0.5) -- (1,6.5,-0.5);
	\fill [white] (0,0,-2) circle (0.01);
\end{tikzpicture}
}
\caption{Division of a bounded three dimensional domain $\om$ into two subdomains according to the values of $c$.}
\label{fig:division}
\end{figure}
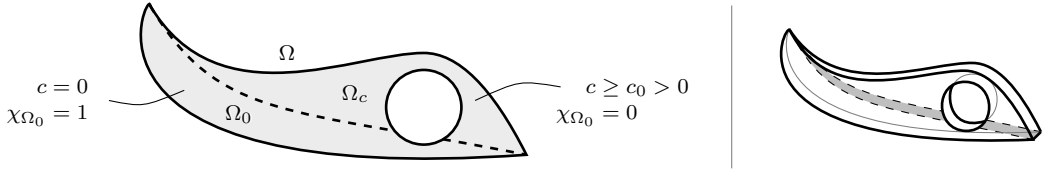

We present the third main result of the paper.

\begin{theo} 
\label{thm:CRD3}
Let $(u,p)\in\hoga\times\d$ be the exact solution of \eqref{eq:pde1}, and assume $c \ge c_0 > 0$ in $\om_c$ and $c-\div b \ge 0$. Let $(\ut,\pt)\in\hoga\times\d$ be arbitrary. Then
\begin{equation} \label{eq:CRD3}
	\norm{f - b\cdot\na\ut + \div\pt}^2_{\om_0} + \frac{1}{2} \normltAi{\pt - \Amat\na\ut}^2
	\leq \norms{(u,p)-(\ut,\pt)}^2 \leq \Mcrdmaj(\ut,\pt)
\end{equation}
holds. If in addition $-\div b \ge 0$ in $\om_c$, we have the improved lower bound
\begin{equation} \label{eq:CRD3_min}
	\Mcrdmin(\ut,\pt) \leq \norms{(u,p)-(\ut,\pt)}^2 .
\end{equation}	
Here
\begin{align*}
	\norms{(u,p)-(\ut,\pt)}^2 := & \, \norm{u-\ut}^2_{c-\div b} + \normltA{\na(u-\ut)}^2
	+ \normltAi{p-\pt}^2 + \norm{b\cdot\na(u-\ut)-\div(p-\pt)}^2_{\chat^{-1}} , \\
	\Mcrdmaj(\ut,\pt) := & \, 2 \left( \norm{f - c\,\ut - b\cdot\na\ut + \div\pt}^2_{\om_c,c^{-1}}
	+ \normltAi{\pt - \Amat\na\ut}^2 \right)
	+ \left( 1+4\MC \right) \norm{f - b\cdot\na\ut + \div\pt}^2_{\om_0} , \\
	\Mcrdmin(\ut,\pt) := & \, \frac{1}{2} \left( \norm{f - c\,\ut - b\cdot\na\ut + \div\pt}^2_{\om_c,c^{-1}} + \normltAi{\pt - \Amat\na\ut}^2 \right)
	+ \norm{f - b\cdot\na\ut + \div\pt}^2_{\om_0} .
\end{align*}
\end{theo}

\begin{proof}
By adding $\chiom u$ to both sides of the problem \eqref{eq:pde1} we obtain
\begin{equation*}
\begin{array}{r@{$\;$}c@{$\;$}l l}
	p & = & \Amat \na u & \quad \textrm{in } \om , \\
	-\div p + b \cdot \na u + \chat\,u & = & f + \chiom u & \quad \textrm{in } \om , \\
	u & = & 0 & \quad \textrm{on } \p\om ,
\end{array}
\end{equation*}
for which we can apply the result of Theorem \ref{thm:CRD1} with reaction coefficient $\chat$ and right hand side $f+\chiom u$:
\begin{multline} \label{eq:CRD3_1}
	\norm{u-\ut}^2_{\chat-\div b} + \normltA{\na(u-\ut)}^2 +
	\normltAi{p-\pt}^2 + \norm{b\cdot\na(u-\ut)-\div(p-\pt)}^2_{\chat^{-1}} \\
	= \norm{f + \chiom u - \chat\,\ut - b\cdot\na\ut + \div\pt}^2_{\chat^{-1}}
	+ \normltAi{\pt - \Amat\na\ut}^2 .
\end{multline}
The rest of the proof concentrates on removing the unknown exact solution from the right hand side. The first term on the right hand side can be written as
\begin{align}
	& \norm{f + \chiom u - \chat\,\ut - b\cdot\na\ut + \div\pt}^2_{\chat^{-1}} \label{eq:CRD3_2} \\
	& = \norm{f + \chiom(u-\ut) - c\,\ut - b\cdot\na\ut + \div\pt}^2_{\chat^{-1}} \nonumber \\
	& = \norm{f - c\,\ut - b\cdot\na\ut + \div\pt}^2_{\om_c,c^{-1}} +
		\norm{f + u-\ut - b\cdot\na\ut + \div\pt}^2_{\om_0} \nonumber \\
	& = \norm{f - c\,\ut - b\cdot\na\ut + \div\pt}^2_{\om_c,c^{-1}} +
		\norm{f - b\cdot\na\ut + \div\pt}^2_{\om_0} + \norm{u-\ut}^2_{\om_0} \nonumber \\
	&	\qquad + 2\Re\scp{f - b\cdot\na\ut + \div\pt}{u-\ut}_{\om_0} . \nonumber
\end{align}
The last term can be estimated from above by
\begin{align}
	2\Re\scp{f - b\cdot\na\ut + \div\pt}{u-\ut}_{\om_0}
	& \leq 2 \norm{f - b\cdot\na\ut + \div\pt}_{\om_0} \norm{u-\ut}_{\om_0} \label{eq:CRD3_3} \\
	& \leq 2 \norm{f - b\cdot\na\ut + \div\pt}_{\om_0} \norm{u-\ut} \nonumber \\
	& \leq 2 \frac{C_F}{\sqrt{\alpha}} \norm{f - b\cdot\na\ut + \div\pt}_{\om_0} \normltA{\na(u-\ut)} \nonumber \\
	& \leq \gamma \MC \norm{f - b\cdot\na\ut + \div\pt}_{\om_0}^2 + \frac{1}{\gamma} \normltA{\na(u-\ut)}^2 , \nonumber
\end{align}
which holds for all $\gamma>0$. By combining \eqref{eq:CRD3_1}--\eqref{eq:CRD3_3} and choosing $\gamma=2$ we obtain
\begin{multline*}
	\norm{u-\ut}^2_{\chat-\div b} + \normltA{\na(u-\ut)}^2
	+ \normltAi{p-\pt}^2 + \norm{b\cdot\na(u-\ut)-\div(p-\pt)}^2_{\chat^{-1}} \\
	\leq \norm{f - c\,\ut - b\cdot\na\ut + \div\pt}^2_{\om_c,c^{-1}}
	+ \left( 1+2\MC \right) \norm{f - b\cdot\na\ut + \div\pt}^2_{\om_0}
	+ \normltAi{\pt - \Amat\na\ut}^2 \\
	+ \norm{u-\ut}^2_{\om_0}
	+ \frac{1}{2} \normltA{\na(u-\ut)}^2 .
\end{multline*}
By moving the last two terms to the left hand side we obtain
\begin{multline} \label{eq:CRD3_4}
	\norm{u-\ut}^2_{c-\div b} + \frac{1}{2} \normltA{\na(u-\ut)}^2
	+ \normltAi{p-\pt}^2 + \norm{b\cdot\na(u-\ut)-\div(p-\pt)}^2_{\chat^{-1}} \\
	\leq \norm{f - c\,\ut - b\cdot\na\ut + \div\pt}^2_{\om_c,c^{-1}}
	+ \left( 1+2\MC \right) \norm{f - b\cdot\na\ut + \div\pt}^2_{\om_0}
	+ \normltAi{\pt - \Amat\na\ut}^2 .
\end{multline}
Since in $\om_0$ we have
\begin{equation} \label{eq:CRD3_5}
	\norm{b\cdot\na(u-\ut)-\div(p-\pt)}^2_{\om_0,\chat^{-1}}
	= \norm{b\cdot\na(u-\ut)-\div(p-\pt)}^2_{\om_0}
	= \norm{f - b\cdot\na\ut + \div\pt}^2_{\om_0}
\end{equation}
we can multiply \eqref{eq:CRD3_4} with any constant without affecting this term in the estimate. We have then arrived at the upper bound in \eqref{eq:CRD3}. The lower bound in \eqref{eq:CRD3} is obtained simply from \eqref{eq:CRD3_5} and inserting $p-\Amat\na u=0$ in
\begin{equation} \label{eq:CRD3_6}
	\normltAi{\pt - \Amat\na\ut}^2
	= \normltAi{\pt-p + \Amat\na(u-\ut)}^2
	\leq 2 \left( \normltAi{\pt-p}^2 + \normltA{\na(u-\ut)}^2 \right) .
\end{equation}
If $-\div b \ge 0$ in $\om_c$, we can also write (note that we have assumed that $c-\div b \ge 0$ holds in the whole domain)
\begin{align}
	\norm{f - c\,\ut - b\cdot\na\ut + \div\pt}^2_{\om_c,c^{-1}}
	& = \norm{ c(u-\ut) + b\cdot\na(u-\ut) + \div(\pt-p)}^2_{\om_c,c^{-1}} \label{eq:CRD3_7} \\
	& \leq 2\left( \norm{u-\ut}^2_{\om_c,c}
		+ \norm{b\cdot\na(u-\ut) + \div(\pt-p)}^2_{\om_c,c^{-1}} \right) \nonumber \\
	& \leq 2\left( \norm{u-\ut}^2_{\om_c,c-\div b}
		+ \norm{b\cdot\na(u-\ut) + \div(\pt-p)}^2_{\om_c,c^{-1}} \right) \nonumber \\
	& \leq 2\left( \normltcb{u-\ut}^2
		+ \norm{b\cdot\na(u-\ut) + \div(\pt-p)}^2_{\om_c,c^{-1}} \right) . \nonumber
\end{align}
The lower bound \eqref{eq:CRD3_min} results then by combining \eqref{eq:CRD3_5}--\eqref{eq:CRD3_7}.
\end{proof}

The above result is more general than the error equalities of the previous section in the sense that there is more freedom for the values of the reaction coefficient $c$ and the convection vector $b$. However, to achieve this we needed to assume boundedness of the domain, and naturally the Friedrichs constant entered the estimates. We also note that the error equalities of the previous section cannot be recovered from Theorem \ref{thm:CRD3}.

\begin{cor} Depending on the reaction coefficient and the convection vector the result of Theorem \ref{thm:CRD3} takes the following forms:
\label{cor:diffcb}
\begin{itemize}
\item[\bf(i)] If $b=0$ the problem \eqref{eq:pde1} becomes the reaction-diffusion problem with reaction limited to $\om_c$. The result holds when $c\ge c_0>0$ in $\om_c$ and it reads as
\begin{align*}
	& \frac{1}{2} \left( \norm{f - c\,\ut + \div\pt}^2_{\om_c,c^{-1}}
		+ \normltAi{\pt - \Amat\na\ut}^2 \right)
		+ \norm{f + \div\pt}^2_{\om_0} \\
	& \qquad \leq \norm{u-\ut}_{\om_c,c}^2 + \normltA{\na(u-\ut)}^2
		+ \normltAi{p-\pt}^2 + \norm{\div(p-\pt)}^2_{\chat^{-1}} \\
	& \qquad\qquad \leq 2 \left( \norm{f - c\,\ut + \div\pt}^2_{\om_c,c^{-1}}
		+ \normltAi{\pt - \Amat\na\ut}^2 \right)
		+ \left( 1+4\MC \right) \norm{f + \div\pt}^2_{\om_0} .
\end{align*}
If $\om_c = \om$ (then $\om_0=\emptyset$) the error equality of Corollary \ref{cor:diffb} holds (and the domain may be unbounded).
\item[\bf(ii)] If $c=0$ ($\om_c=\emptyset, \om_0=\om$) the problem \eqref{eq:pde1} becomes the convection-diffusion problem. The result holds when $-\div b \ge 0$, and it reads as
\begin{align*}
	& \frac{1}{2} \normltAi{\pt - \Amat\na\ut}^2 + \norm{f - b\cdot\na\ut + \div\pt}^2 \\
	& \qquad \leq \normltb{u-\ut}^2 + \normltA{\na(u-\ut)}^2 + \normltAi{p-\pt}^2
		+ \norm{b\cdot\na(u-\ut)-\div(p-\pt)}^2 \\
	& \qquad\qquad \leq 2\normltAi{\pt - \Amat\na\ut}^2
		+ \left( 1+4\MC \right) \norm{f - b\cdot\na\ut + \div\pt}^2 .
\end{align*}
If $-\div b \ge b_0>0$ the error equality of Corollary \ref{cor:diffc} holds.
\item[\bf(iii)] If $c=b=0$ ($\om_c=\emptyset, \om_0=\om$) the problem \eqref{eq:pde1} becomes the diffusion equation and we have the estimate
\begin{align*}
	& \frac{1}{2} \normltAi{\pt - \Amat\na\ut}^2 + \norm{f + \div\pt}^2 \\
	& \qquad \leq \normltA{\na(u-\ut)}^2 + \normltAi{p-\pt}^2 + \norm{\div(p-\pt)}^2 \\
	& \qquad\qquad \leq 2\normltAi{\pt - \Amat\na\ut}^2
		+ \left( 1+4\MC \right) \norm{f + \div\pt}^2 ,
\end{align*}
which is similar to the two-sided estimate derived (in the real case) in \cite{repinsautersmolianskiaposttwosideell}.
\end{itemize}
\end{cor}

\begin{rem} \mbox{}
\begin{itemize}
\item[\bf(i)] It is easy to verify (in the case $-\div b \ge 0$ in $\om_c$) that
\begin{equation*}
	\sqrt{\frac{\Mcrdmaj(\ut,\pt)}{\Mcrdmin(\ut,\pt)}} \leq \max\left\{2,\sqrt{1+4\MC}\right\} .
\end{equation*}
\item[\bf(ii)] By calculating estimates for the solution operator norm the two-sided bounds exposed in this section can be normalized in the same way as in the previous section (one way to obtain estimates for the solution operator norm is setting $\ut=\pt=0$ in Theorem \ref{thm:CRD3}).
\item[\bf(iii)] The upper bound in Theorem \ref{eq:CRD3} can be improved if one leaves the Young constant $\gamma$ in the estimate (in \eqref{eq:CRD3_3} this constant was chosen to be 2). Then the upper bound would contain this constant which could be chosen such that it minimizes the value of the upper bound.
\end{itemize}
\end{rem}


\section{Further Remarks}

In this section we briefly comment on deriving error estimates for the primal variable and error indication properties for mixed approximations. For brevity we limit ourselves to the problem \eqref{eq:pde1} and the error equality of Theorem \ref{thm:CRD1}. In the following we will utilize just the following inequalities:
\begin{align} 
	\normltAi{\pt - \Amat\na\ut}^2
	& = \normltAi{\pt-p + \Amat\na(u-\ut)}^2 \label{eq:low1} \\
	& \leq 2 \left( \normltAi{\pt-p}^2 + \normltA{\na(u-\ut)}^2 \right) \nonumber , \\
	\normltci{f - c\,\ut - b\cdot\na\ut + \div\pt}^2
	& = \normltci{c(u-\ut) + b\cdot\na(u-\ut) - \div(p-\pt)}^2 \label{eq:low2} \\
	& \leq 2 \left( \normltc{u-\ut}^2 + \normltci{b\cdot\na(u-\ut) - \div(p-\pt)}^2 \right) \nonumber .
\end{align}

Theorem \ref{thm:CRD1} reads as
\begin{multline} \label{eq:erroreq}
	\normltcb{u-\ut}^2 + \normltA{\na(u-\ut)}^2 +
	\normltAi{p-\pt}^2 + \normltci{b\cdot\na(u-\ut)-\div(p-\pt)}^2 \\
	= \normltci{f - c\,\ut - b\cdot\na\ut + \div\pt}^2
	+ \normltAi{\pt - \Amat\na\ut}^2 .
\end{multline}
The first obvious estimate results from simply omitting the last two terms on the left hand side. Then the left hand side becomes independent of $\pt$, and so it becomes arbitrary on the right hand side, and we have
\begin{equation*}
	\forall \phi\in\d
	\qquad
	\normltcb{u-\ut}^2 + \normltA{\na(u-\ut)}^2
	\leq \normltci{f - c\,\ut - b\cdot\na\ut + \div\phi}^2
	+ \normltAi{\phi - \Amat\na\ut}^2 ,
\end{equation*}
where we have changed $\pt$ to $\phi$ to emphasize that it is now an arbitrary function from $\d$. As is usual for problems with a convective term, such a form of a functional upper bound may unfortunately not be sharp: to get an idea of the over-estimation we set $\phi = p$; then the r.h.s. becomes $\normltcb{u-\ut}^2 + \normltci{b\cdot\na(u-\ut)}^2 + \normltA{\na(u-\ut)}^2$. In \cite{kuzminhannukainenapost} the authors demonstrated that sharp upper bounds can be derived as well. The estimate from \cite{kuzminhannukainenapost} contains two free functions (in contrast to the one free function in the above estimate) with respect to which the upper bound must be minimized in order to obtain good estimates of the global primal error.

Next we derive estimates for the $\lt$-part of the primal error: by estimating the left hand side of \eqref{eq:erroreq} from below by (omitting the last term and) using \eqref{eq:low1} we obtain
\begin{equation} \label{eq:ubound1}
	\forall \phi\in\d
	\qquad
	\normltcb{u-\ut}^2
	\leq \normltci{f - c\,\ut - b\cdot\na\ut + \div\phi}^2
	+ \frac{1}{2} \normltAi{\phi - \Amat\na\ut}^2 .
\end{equation}
If $-\div b \ge b_0 > 0$ we can apply both \eqref{eq:low1} and \eqref{eq:low2} to the error equality \eqref{eq:erroreq} and obtain
\begin{equation} \label{eq:ubound2}
	\forall \phi\in\d
	\qquad
	\normltb{u-\ut}^2
	\leq \frac{1}{2} \left( \normltci{f - c\,\ut - b\cdot\na\ut + \div\phi}^2
	+ \normltAi{\phi - \Amat\na\ut}^2 \right) .
\end{equation}

As mentioned in the introduction, mixed methods usually produce approximations with less regularity than assumed in this paper. A first simple idea might be to post-process the approximation to obtain upper bounds for errors. For example, assume that $\ut\in\lt$ and that we have a post-processing operator $G:\lt\to\hoga$ at our disposal (for example, for piecewise constant $\ut$ one could use standard nodal averaging). By using \eqref{eq:ubound1} we then have the estimate
\begin{align*}
	\normltcb{u-\ut}
	& \leq \normltcb{u-G(\ut)} + \normltcb{G(\ut)-\ut} \\
	& \leq \sqrt{\normltci{f - c\,G(\ut) - b\cdot\na G(\ut) + \div\phi}^2
	+ \frac{1}{2} \normltAi{\phi - \Amat\na G(\ut)}^2}
	+ \normltcb{G(\ut)-\ut}
\end{align*}
for any $\phi\in\d$. If $-\div b \ge b_0 > 0$ we can use \eqref{eq:ubound2} in an similar fashion to obtain
\begin{align*}
	\normltb{u-\ut}
	& \leq \normltb{u-G(\ut)} + \normltb{G(\ut)-\ut} \\
	& \leq \frac{1}{\sqrt{2}} \sqrt{\normltci{f - c\,G(\ut) - b\cdot\na G(\ut) + \div\phi}^2
	+ \normltAi{\phi - \Amat\na G(\ut)}^2}
	+ \normltb{G(\ut)-\ut}
\end{align*}
for any $\phi\in\d$.

Up until now only global values of the error have been discussed. However, estimating the error distribution of an approximation is an important task as well since this is essential for adaptive mesh refinement. Let $\elems$ denote a discretization of the domain $\om$ into a mesh of non-overlapping elements $\elem$. Note that we assume $\bigcup_{\elem\in\elems} \ol{\elem} = \ol{\Omega}$, i.e., in particular that the boundary of $\Omega$ is exactly represented by the mesh. This is necessary in order to have conforming approximations in the first place: they must satisfy exactly the imposed boundary conditions. The upper bound $\Mcrd$ from Theorem \ref{thm:CRD1} generates the following error indicator:
\begin{equation*}
	\eta_\elem(\ut,\pt) :=
	\sqrt{\norm{f - c\,\ut - b\cdot\na\ut + \div\pt}_{\elem,c^{-1}}^2
	+ \norm{\pt - \Amat\na\ut}^2_{\elem,\Amat^{-1}}} ,
\end{equation*}
whose purpose is to indicate the exact error distribution
\begin{equation*}
	e_\elem(\ut,\pt) :=
	\sqrt{ \norm{u-\ut}_{\elem,c-\div b}^2
	+ \norm{\na(u-\ut)}_{\elem,\Amat}^2
	+ \norm{p-\pt}_{\elem,\Amat^{-1}}^2
	+ \norm{b\cdot\na(u-\ut)-\div(p-\pt)}_{\elem,c^{-1}}^2 } .
\end{equation*}
In the following
\begin{equation*}
	\eta:=\sqrt{\sum_{\elem\in\elems} \eta_\elem^2}
	\qquad \textrm{and} \qquad
	e:=\sqrt{\sum_{\elem\in\elems} e_\elem^2} .
\end{equation*}
The often cited global requirement for error indicators reads as $\ul{C}\eta \leq e \leq \ol{C} \eta$, where $\ul{C},\ol{C}>0$ are called the global efficiency constant and global reliability constant, respectively. Obviously in this case $\ul{C} = \ol{C} = 1$, which is the best case possible. What remains to be checked are the values of the local efficiency constants $C_T>0$ in the local requirement $C_\elem \eta_\elem \leq e_\elem$. In the case $-\div b \ge 0$ this is particularly simple: since the inequalities \eqref{eq:low1} and \eqref{eq:low2} result simply from applying the triangle inequality and the Young inequality, they hold also in any subdomain, and directly give $\eta_\elem^2 \leq 2 e_\elem^2$. This shows that $C_\elem = 1/\sqrt{2}>0.7$ for any element $\elem$ regardless of shape or size.

Finally we note that the (non-normalized) error equalities and bounds exposed in this paper hold also for non-homogenous Dirichlet boundary conditions provided that the approximation $\ut$ satisfies the boundary condition exactly: in the proofs we require that $u-\ut$ vanishes on the boundary.


\bibliographystyle{plain} 
\bibliography{biblio}


\appendix

\section{Proof of Identity \eqref{eq:identity1}}
\label{app:identity1}

For an arbitrary domain $\om$ we can define $\hoga:=\ol{\ciga}^{\ho}$. Let $v,w\in\hoga$. Because $\ciga$ is dense in $\hoga$, the following two sequences exist
\begin{equation*}
	(w_n),(v_n) \subset \ciga(\om)
	\qquad w_n \rightarrow w, \quad v_n \rightarrow v
	\qquad \textrm{in} \, \ho .
\end{equation*}
It is easy to show that
\begin{equation*}
	v_n w_n \rightarrow v \, w,
	\qquad \na(v_n w_n) \rightarrow \na(v \, w)
	\qquad \textrm{in} \, \lo .
\end{equation*}
Let $b$ be a vector valued function from $\li$ such that $\div b \in \li$. We then have
\begin{equation*}
	\scplt{v \, w}{\div b}
	\leftarrow
	\scplt{v_n w_n}{\div b} = - \scplt{\na(v_n w_n)}{b}
	\rightarrow
	- \scplt{\na(v \, w)}{b} .
\end{equation*}
From this we obtain
\begin{equation*}
	\scplt{v \, w}{\div b} = - \scplt{\na(v \, w)}{b} = - \scplt{w \na v}{b} - \scp{v \na w}{b} ,
\end{equation*}
which gives \eqref{eq:identity1} after re-arranging the functions inside the inner products.


\section{Norm Equivalences}
\label{app:normequiv}

We show that the norms $\normm{\cdot}$ and $\normmm{\cdot}$ are equivalent to weighed $\ho \times \d$ -norms. Let $(x,y) \in \ho \times \d$ be arbitrary.

\begin{theo} \label{thm:normeq1}
Let the assumptions on the material coefficients of Remark \ref{rem:isometry}(i) hold. We have
\begin{equation} \label{eq:normeq1}
	\left( \max \left\{ 2, 1+2\frac{\normli{b}^2}{\alpha c_0} \right\} \right)^{-1} \normm{(x,y)}^2
	\leq \normm{(x,y)}_{\ho \times \d}^2
	\leq 2 \left( 1+\frac{\normli{b}^2}{\alpha c_0} \right) \normm{(x,y)}^2 ,
\end{equation}
where $\normm{(x,y)}_{\ho \times \d}^2 := \normltcb{x}^2 + \normltA{\na x}^2 + \normltAi{y}^2 + \normltci{\div y}^2$.
\end{theo}

\begin{proof}
First we estimate
\begin{equation} \label{eq:normeq1_1}
	\normltci{b\cdot\na x}^2
	\leq \frac{1}{c_0} \normlt{b\cdot\na x}^2
	\leq \frac{\normli{b}^2}{c_0} \normlt{\na x}^2
	\leq \frac{\normli{b}^2}{\alpha c_0} \normltA{\na x}^2 ,
\end{equation}
where $\normli{b}^2 := \sum_j \normli{b_j}^2$.
With \eqref{eq:normeq1_1} we have the lower bound in \eqref{eq:normeq1}:
\begin{align*}
	\normm{(x,y)}^2
	& = \normltcb{x}^2 + \normltA{\na x}^2 + \normltAi{y}^2 + \normltci{b\cdot\na x - \div y}^2 \\
	& \leq \normltcb{x}^2 + \normltA{\na x}^2 + \normltAi{y}^2 + 2\normltci{b\cdot\na x}^2 + 2\normltci{\div y}^2 \\
	& \leq \normltcb{x}^2 + \normltA{\na x}^2 + \normltAi{y}^2 + 2\frac{\normli{b}^2}{\alpha c_0} \normltA{\na x}^2 + 2\normltci{\div y}^2 \\
	& \leq \max \left\{ 2, 1+2\frac{\normli{b}^2}{\alpha c_0} \right\} \left( \normltcb{x}^2 + \normltA{\na x}^2 + \normltAi{y}^2 + \normltci{\div y}^2 \right) .	
\end{align*}
With \eqref{eq:normeq1_1} we also have
\begin{align*}
	\normm{(x,y)}^2
	& = \normltcb{x}^2 + \normltA{\na x}^2 + \normltAi{y}^2 + \normltci{b\cdot\na x - \div y}^2 \\
	& \ge \normltcb{x}^2 + \normltA{\na x}^2 + \normltAi{y}^2 - \normltci{b\cdot\na x}^2 + \frac{1}{2} \normltci{\div y}^2 \\
	& \ge \normltcb{x}^2 + \normltA{\na x}^2 + \normltAi{y}^2 - \frac{\normli{b}^2}{\alpha c_0} \normltA{\na x}^2 + \frac{1}{2} \normltci{\div y}^2 \\
	& \ge \normltcb{x}^2 + \normltA{\na x}^2 + \normltAi{y}^2 - \frac{\normli{b}^2}{\alpha c_0} \normm{(x,y)}^2 + \frac{1}{2} \normltci{\div y}^2 .
\end{align*}
By re-arranging the terms we obtain
\begin{align*}
	\left( 1+\frac{\normli{b}^2}{\alpha c_0} \right) \normm{(x,y)}^2
	& \ge \normltcb{x}^2 + \normltA{\na x}^2 + \normltAi{y}^2 + \frac{1}{2} \normltci{\div y}^2 \\
	& \ge \frac{1}{2} \left( \normltcb{x}^2 + \normltA{\na x}^2 + \normltAi{y}^2 + \normltci{\div y}^2 \right) ,
\end{align*}
and we have the upper bound in \eqref{eq:normeq1}.
\end{proof}

\begin{theo} \label{thm:normeq2}
Let the assumptions on the material coefficients of Remark \ref{rem:isometry}(ii) hold. We have
\begin{equation} \label{eq:normeq2}
	\left( \max \left\{ 2, 1+2\frac{\normli{b}^2}{\alpha c_0} \right\} \right)^{-1} \normmm{(x,y)}^2
	\leq \normmm{(x,y)}_{\ho \times \d}^2
	\leq 2 \left( 1+\frac{\normli{b}^2}{\alpha c_0} \right) \normmm{(x,y)}^2 ,
\end{equation}
where $\normmm{(x,y)}_{\ho \times \d}^2 := \normltc{x}^2 + \normltA{\na x}^2 + \normltAi{y}^2 + \normltcbi{\div y}^2$.
\end{theo}

\begin{proof}
First we estimate
\begin{equation} \label{eq:normeq2_1}
	\normltAi{b\,x}^2
	\leq \frac{1}{\alpha} \normlt{b\,x}^2
	\leq \frac{\normli{b}^2}{\alpha} \normlt{x}^2
	\leq \frac{\normli{b}^2}{\alpha c_0} \normltc{x}^2 .
\end{equation}
With \eqref{eq:normeq2_1} we have the lower bound in \eqref{eq:normeq2}:
\begin{align*}
	\normmm{(x,y)}^2
	& = \normltc{x}^2 + \normltA{\na x}^2 + \normltAi{y+b\,x}^2 + \normltcbi{\div y}^2 \\
	& \leq \normltc{x}^2 + \normltA{\na x}^2 + 2\normltAi{y}^2 + 2\normltAi{b\,x}^2 + \normltcbi{\div y}^2 \\
	& \leq \normltc{x}^2 + \normltA{\na x}^2 + 2\normltAi{y}^2 +  2 \frac{\normli{b}^2}{\alpha c_0} \normltc{x}^2 + \normltcbi{\div y}^2 \\
	& \leq \max\left\{ 2, 1+2 \frac{\normli{b}^2}{\alpha c_0} \right\} \left( \normltc{x}^2 + \normltA{\na x}^2 + \normltAi{y}^2 + \normltcbi{\div y}^2 \right) .
\end{align*}
With \eqref{eq:normeq2_1} we also have
\begin{align*}
	\normmm{(x,y)}^2
	& = \normltc{x}^2 + \normltA{\na x}^2 + \normltAi{y+b\,x}^2 + \normltcbi{\div y}^2 \\
	& \ge \normltc{x}^2 + \normltA{\na x}^2 + \frac{1}{2} \normltAi{y}^2 - \normltAi{b\,x}^2 + \normltcbi{\div y}^2 \\
	& \ge \normltc{x}^2 + \normltA{\na x}^2 + \frac{1}{2} \normltAi{y}^2 - \frac{\normli{b}^2}{\alpha c_0} \normltc{x}^2 + \normltcbi{\div y}^2 \\
	& \ge \normltc{x}^2 + \normltA{\na x}^2 + \frac{1}{2} \normltAi{y}^2 - \frac{\normli{b}^2}{\alpha c_0} \normmm{(x,y)}^2 + \normltcbi{\div y}^2 .
\end{align*}
By re-arranging the terms we obtain
\begin{align*}
	\left( 1+ \frac{\normli{b}^2}{\alpha c_0} \right)\normmm{(x,y)}^2
	& \ge \normltc{x}^2 + \normltA{\na x}^2 + \frac{1}{2} \normltAi{y}^2 + \normltcbi{\div y}^2 \\
	& \ge \frac{1}{2} \left( \normltc{x}^2 + \normltA{\na x}^2 + \normltAi{y}^2 + \normltcbi{\div y}^2 \right) ,
\end{align*}
and we have the upper bound in \eqref{eq:normeq2}.
\end{proof}

For completeness, we also remark the following.

\begin{rem} \label{rem:bounded}
In the following cases we need to assume that the domain is bounded:
\begin{itemize}
\item[\bf(i)] If $c-\div b=0$, the norm $\normm{\cdot}$ becomes
\begin{equation*}
	\normm{(x,y)}^2 = \normltA{\na x}^2 + \normltAi{y}^2 + \normltci{b\cdot\na x - \div y}^2 ,
\end{equation*}
and for proper error control, we need to assume that $\Omega$ is bounded so that the Friedrichs inequality \eqref{eq:Cf} holds: then $\normltA{\na x}$ is equivalent to the weighed $\ho$-norm $\sqrt{\normlt{x}^2 + \normltA{\na x}^2}$. Note that the norm equivalence of Theorem \ref{thm:normeq1} holds also in this case.
\item[\bf(ii)] If $c=0$, the norm $\normmm{\cdot}$ becomes
\begin{equation*}
	\normmm{(x,y)}^2 = \normltA{\na x}^2 + \normltAi{y+b\,x}^2 + \normltbi{\div y}^2 ,
\end{equation*}
and again, for proper error control, we need to assume that $\Omega$ is bounded so that the Friedrichs inequality \eqref{eq:Cf} holds. Note that the norm equivalence of Theorem \ref{thm:normeq2} no longer holds. However, with \eqref{eq:Cf} we can estimate
\begin{equation*}
	\normltAi{b\,x}^2
	\leq \frac{1}{\alpha} \normlt{b\,x}^2
	\leq \frac{\normli{b}^2}{\alpha} \normlt{x}^2
	\leq \frac{\normli{b}^2 \Cf^2}{\alpha} \normlt{\na x}^2
	\leq \frac{\normli{b}^2 \Cf^2}{\alpha^2} \normltA{\na x}^2 ,
\end{equation*}
and by proceeding as in the proof of Theorem \ref{thm:normeq2} we obtain a similar norm equivalence.
\end{itemize}
\end{rem}

\end{document}